\documentclass[DIV12,11pt]{scrartcl}

\usepackage{amsthm,amssymb,amsmath}
\usepackage[noblocks]{authblk}

\usepackage{amsmath}
\usepackage{amsopn}
\usepackage{amssymb}
\usepackage{latexsym}
\usepackage{amsfonts}
\usepackage{amsthm}
\usepackage[all]{xy}
\usepackage{tikz}
\usetikzlibrary{decorations.markings}
\tikzstyle{vertex}=[circle, draw, inner sep=0pt, minimum size=6pt]

\setkomafont{title}{\normalfont \LARGE \bfseries}
\setkomafont{section}{\normalfont \Large \bfseries \boldmath}
\setkomafont{subsection}{\normalfont \large \bfseries}
\setkomafont{subsubsection}{\normalfont \normalsize \bfseries}
\setkomafont{descriptionlabel}{\itshape}
\setkomafont{paragraph}{\normalfont \bfseries \boldmath}

\newtheorem{thm}{Theorem}[section]
\newtheorem{pro}[thm]{Proposition}
\newtheorem{lem}[thm]{Lemma}
\newtheorem{cla}[thm]{Claim}

\newtheorem{cor}[thm]{Corollary}

\theoremstyle{definition}
\newtheorem{obs}[thm]{Observation}

\newtheorem{rem}[thm]{Remark}
\newtheorem{exa}[thm]{Example}
\newtheorem{defn}[thm]{Definition}

\newtheorem{conj}[thm]{Conjecture}

\newcommand{\een}{\end{enumerate}}
\newcommand{\blem}{\begin{lem}}
\newcommand{\elem}{\end{lem}}
\newcommand{\bcl}{\begin{cla}}
\newcommand{\ecl}{\end{cla}}
\newcommand{\ethm}{\end{thm}}
\newcommand{\bpr}{\begin{pro}}
\newcommand{\epr}{\end{pro}}
\newcommand{\bco}{\begin{cor}}
\newcommand{\eco}{\end{cor}}
\newcommand{\bcon}{\begin{conj}}
\newcommand{\econ}{\end{conj}}
\newcommand{\bde}{\begin{defn}}
\newcommand{\ede}{\end{defn}}
\newcommand{\bex}{\begin{exa}}
\newcommand{\eexa}{\end{exa}}
\newcommand{\bobs}{\begin{obs}}
\newcommand{\eobs}{\end{obs}}
\newcommand{\bexe}{\begin{exe}}
\newcommand{\eexe}{\end{exe}}

\begin{document}

\title{Weighted Coxeter graphs and generalized geometric representations of Coxeter Groups}
\author[1]{Vadim Bugaenko}
\author[2]{Yonah Cherniavsky\footnote{The second and third authors were supported by
the Swiss National Science Foundation.}\footnote{{\it Corresponding author}: yonahch@ariel.ac.il}}
\author[3]{Tatiana Nagnibeda}
\author[4]{Robert Shwartz}

\affil[1]{Department of Computer Science and Mathematics, Ariel University, Israel; vdmbgnk@gmail.com}
\affil[2]{Department of Computer Science and Mathematics, Ariel University, Israel; yonahch@ariel.ac.il}
\affil[3]{Department of Mathematics, University of Geneva, Switzerland; tatiana.smirnova-nagnibeda@unige.ch}
\affil[4]{Department of Computer Science and Mathematics, Ariel University, Israel; robertsh@ariel.ac.il}


\date{\vspace*{-2em}}
\maketitle
\begin{abstract} 
\textbf{Abstract.} We introduce the notion of weighted Coxeter graph and associate to it a certain generalization of the standard geometric representation of a Coxeter group. We prove sufficient conditions for faithfulness and non-faithfulness of such a representation. In the case when the weighted Coxeter graph is balanced we discuss how the generalized geometric representation is related to the numbers game played on the Coxeter graph.
\\ \\
\textbf{Keywords:} Coxeter groups; Coxeter graphs; geometric representation; numbers game; signed graphs; weighted graphs; balanced functions on graphs.
\end{abstract}

\section{Introduction and preliminaries}
The theory of Coxeter groups is connected with different fields of mathematics: algebra, geometry, combinatorics, graph theory. In this work we study certain linear representations of Coxeter groups and relate them to the corresponding Coxeter graphs.

From the algebraic point of view a Coxeter system $(W,S)$ is a group $W$ with the set of generators $S=\{s_1,s_2,...,s_n\}$ and relations $s_i^2=1$, $(s_is_j)^{m_{ij}}=1$, where $2\leqslant m_{ij}\in\mathbb N$ or $m_{ij}=\infty$, the latter means  that there is no relation between $s_i$ and $s_j$.

Geometrically, Coxeter groups can be viewed as discrete groups generated by orthogonal reflections in a vector space with a pseudo-euclidian metric.

A Coxeter group can also be defined by its Coxeter graph. Its vertex set coincides with the set $S$ of Coxeter generators. Two vertices $s_i$ and $s_j$ are not connected if $m_{ij}=2$, are connected by an unlabeled edge if $m_{ij}=3$ and are connected by an edge labeled by $m_{ij}$ if $m_{ij}>3$. Edges labeled by $m_{ij}$ with $m_{ij}>3$ are called {\it multiple edges}.
A Coxeter group is called {\it simply laced} if its Coxeter graph does not have multiple edges, i.e., $m_{ij}$ equals to 2 or 3 for all $i\neq j$.

 A combinatorial model of the Coxeter group can be given by the so called Mozes' {\bf numbers game} which we briefly describe now. Consider a graph with vertices indexed by 1,2,..., $n$ and numbers $x_j$ associated to its vertices. The column vectors $(x_1,x_2,...,x_n)^t$ are called configurations or game positions. The (legal) ``moves" in the game (which are also called ``firings") are local rearrangements of the assigned values at a chosen node and its neighbors: the move which corresponds to a vertex $j$ consists of adding the value $x_j$ multiplied by a certain positive weight $k_{ij}$ to each value $x_i$, i.e., $x_i:=x_i+k_{ij}x_j$, where the vertices $j$ and $i$ are neighbors, and then reversing the sign, i.e., $x_j:=-x_j$. Using Weyl groups and Kac-Moody algebras, see~\cite{Bourbaki} and~\cite{Kac} for definitions and other details, Mozes in~\cite{M} gave an algebraic characterization
of the initial positions giving rise to finite games and proved that for those the number of steps and the final
configuration do not depend on the moves of the player. It was developed as a generalization of the following problem of the International Mathematical Olympiad (IMO) in 1986: five integers with positive sum are assigned to the vertices of a pentagon. If there is at least
one negative number, the player may pick one of them, say $y$, add it to its two neighbors $x$ and $z$, and then reverse the
sign of $y$. The game terminates when all the numbers are nonnegative. Prove that this game must always terminate. Several solutions to this problem can be found in~\cite{WR}. Eriksson and Bj\"orner found deep connections between the numbers game and Coxeter groups: taking the positive weights $k_{ij}$ satisfying $k_{ij}k_{ji}=4\cos^2\left(\pi/m_{ij}\right)$ where $m_{ij}$ is the minimal number such that for Coxeter generators $s_i$ and $s_j$ there is the relation $\left(s_is_j\right)^{m_{ij}}=1$ and $k_{ij}k_{ji}\geqslant4$ when the element $s_is_j$ has infinite order, the numbers game becomes a combinatorial model of the Coxeter group, where group elements correspond to positions and reduced decompositions correspond to legal play sequences, see Chapter 4 in~\cite{BB} and~\cite{E} for details.

Any Coxeter group has the canonical linear representation called the {\bf geometric representation} defined in the following way. The dimension of this representation equals to the number of Coxeter generators. For each pair of generators $(s_i,s_j)$, which do not commute, choose a {\bf positive} number $k_{ij}$ such that $k_{ij}k_{ji}=4\cos^2\left(\pi/m_{ij}\right)$; if $m_{ij}=\infty$, put $k_{ij}=k_{ji}=2$.  Each generator $s_i$ is mapped to the matrix $\sigma_i$ which differs from the identity matrix only in its $i$-th row. The diagonal element in the position $(i,i)$ is $-1$, and for $i\neq j$ the entry in the position $(i,j)$ equals to $k_{ij}$; if the generators $s_i$ and $s_j$ commute, then the entry in the position $(i,j)$ is zero. Numbers $k_{ij}$ and $k_{ji}$ may be different. If all the numbers $m_{ij}$ are 2,3,4,6 or $\infty$ it is possible to choose all $k_{ij}$ integers. We emphasize that the numbers $k_{ij}$ are positive because it is crucial for the following theorem which is well known (see \cite{BB}):\\
\textbf{Theorem.} Matrices $\sigma_i$ satisfy the Coxeter relations, i.e., the mapping $s_i\mapsto\sigma_i$ is a representation. The representation $s_i\mapsto\sigma_i$ is faithful.
The representation $s_i\mapsto\sigma_i$ is called the {\bf geometric representation}.

If for any $i,j$, $k_{ij}=k_{ji}=2\cos\left(\pi/m_{ij}\right)$, then the representation $s_i\mapsto\sigma_i$ is called the standard geometric representation. When the group is simply laced, all non-zero entries of matrices $\sigma_i$ of the standard geometric representation are $\pm1$.

The connection between the numbers game and the standard geometric representation is as follows: it is easy to see that a move associated with the vertex $i$ in the described above numbers game is the action of the matrix $\sigma_i^t$ on the column vector of configuration $(x_1,x_2,...,x_n)^t$.

This paper is devoted to a natural generalization of the objects defined above, obtained by allowing the numbers $k_{ij}$ which are used in the definition of the geometric representation and numbers game to take not necessarily real positive values. We define a corresponding version of the numbers game and a generalization of the standard geometric representation and give sufficient conditions for its faithfulness and non-faithfulness. In the next section we give main definitions and formulate the results. We freely use the standard terminology and results on Coxeter groups, Coxeter graphs etc. applied in the monographs~\cite{Bourbaki},~\cite{Kac},~\cite{H},~\cite{BB}.

\section{Weighted Coxeter Graph and the Corresponding Representation}
In this paper we shall consider geometric representations and number games associated to {\it weighted Coxeter graphs}. Let us associate weights to the edges of Coxeter graphs as follows.

 Denote by $\bar E$ the set of directed edges of the graph $\Gamma=(V,E)$, i.e.,
$$
\left\{s_i,s_j\right\}\in E\Longleftrightarrow\left(s_i,s_j\right)\,,\,\left(s_j,s_i\right)\in\bar E\,.
$$
\bde\label{lwf}
Let $\Gamma=(V,E)$ be the Coxeter graph of the  Coxeter system $(W,S)$.
Define a {\bf legal weight function} to be a function
$$
f\,:\,\bar E\rightarrow \mathbb C\backslash \{0\}\quad\textrm{such that}\quad f\left(\left(s_i,s_j\right)\right)=\left( f\left(\left(s_j,s_i\right)\right)\right)^{-1}\,.
$$
\ede
Notice that the complex numbers $\left\{f\left(\left(s_i,s_j\right)\right)\right\}$ and the integers $\left\{m_{ij}\right\}$ (when $m_{ij}>3$) are two independent sets of weights on the edges of our graph.

\bde
We say that the triple $\Gamma_f=(V,E,f)$ is a {\bf weighted Coxeter graph} if $\Gamma=(V,E)$ is a Coxeter graph and $f$ is a legal weight function on its edges.
\ede
\bde
 For any path in a weighted Coxeter graph we can define the {\bf weight of the path} as the product of all weights of the edges of which this path consists. We call a weighted Coxeter graph {\bf balanced} if the weight of any closed path in this graph is equal to 1.
\ede
\begin{rem}
The notion of balanced legal functions on directed edges of a graph is introduced in the literature under different names. Thus, for example, in \cite{BHN} the set of logarithms of real-valued balanced legal functions is exactly $Im(d)$ and in \cite{KY}, in somewhat different language, that set is referred-to as the set of consistent graphs, see also~\cite{CGL2}. In~\cite{T} such functions have been introduced under the name ``color-coboundaries". Also they appear in literature under the name ``tensions", see, for example,~\cite{K}. In a rather common terminology introduced by Zaslavsky,~\cite{Z}, a pair of graph and such a balanced function on edges of a graph is called a balanced gain graph.
\end{rem}
\bde
In the particular case when the graph $\Gamma=(V,E)$ is the Coxeter graph of a simply laced Coxeter system and the legal weight function takes values in the cyclic group of order two, i.e., the weights on the edges are $\pm1$, we say that $\Gamma_f=(V,E,f)$ is a {\bf signed Coxeter graph}.
\ede
\begin{rem} The weight of any path in a balanced graph depends only on its start and end vertices, the weight does not depend on the path itself.
\end{rem}

The canonical construction of the geometric representation can be generalized for a weighted Coxeter graph $\Gamma_f=(V,E,f)$ in the following way.
\bde Let us construct the mapping: the generator $s_i$ is mapped to the $n\times n$ matrix $\omega_i$ which differs from the identity matrix only by the $i$-th row. The $i$-th row of the matrix $\omega_i$ has $-1$ at the position $(i,i)$, it has $f\left(\left(s_i,s_j\right)\right)\cdot2\cos\left(\pi/m_{ij}\right)$ in the position $(i,j)$ when the node $s_j$ is connected to the node $s_i$, and it has 0 in the position $(i,j)$ when the nodes $s_j$ and $s_i$ are not connected by an edge. Thus, we defined the mapping $\mathcal{R}_{\Gamma,f}\,:\,S\rightarrow GL_n\left(\mathbb C\right)$, $\mathcal{R}_{\Gamma,f}(s_i)=\omega_i$.
\ede
Note that if we put $k_{ij}:=f\left(\left(s_i,s_j\right)\right)\cdot2\cos\left(\pi/m_{ij}\right)$, then the equality $k_{ij}k_{ji}=4\cos^2\left(\pi/m_{ij}\right)$ still holds since $f\left(\left(s_i,s_j\right)\right)\cdot f\left(\left(s_j,s_i\right)\right)=1$ even though $k_{ij}\in\mathbb C$ are not necessarily positive.

Notice that in the simply laced case one can extend the construction and consider legal weights with values in an arbitrary abelian group $G$; the matrices $\omega_i$ in this case will be over the group ring $\mathbb C[G]$.
\bex Consider for example the weighted Coxeter graph of the symmetric group $S_4$:
\xygraph{
!{<0cm,0cm>;<1cm,0cm>:<0cm,1cm>::}
!{(5,0.05) }*{s_{1}\,\,}="1"
!{(6,0.05) }*{\,\,\,s_{2}\,}="2"
!{(7,0.05)}*{\,\,\,s_{3}\,}="3"
"1"-"2"^{a}
"2"-"3"^{b}
} $\quad$

$$s_1\mapsto\omega_1=\begin{pmatrix}
-1 &a &0 \\
0 &1 &0 \\
0 &0 &1
\end{pmatrix}\,,\, s_2\mapsto\omega_2=\begin{pmatrix}
1 &0 &0\\
a^{-1} &-1 &b\\
0 &0 &1
\end{pmatrix}\, , \,s_3\mapsto\omega_3=\begin{pmatrix}
1 &0 &0 \\
0 &1 &0 \\
0 &b^{-1} &-1
\end{pmatrix}\,.$$
Notice that when $a=b=1$, we get the standard geometric representation of $S_4$.
\eexa

It can be directly checked that the following proposition holds:
\bpr The mapping $\mathcal{R}_{\Gamma,f}(s_i)=\omega_i$ can be extended to a group homomorphism $W\rightarrow\mathbb C\setminus\{0\}$.
\epr
In other words, the matrix group $\Omega=\langle \omega_1,\omega_2,...,\omega_n\rangle$ is isomorphic to some quotient, may be proper, of the Coxeter group $W$.
The standard geometric representation is a particular case of the representation $\mathcal{R}_{\Gamma,f}$ when the function $f$ maps every edge to 1, i.e., the weight of any edge is 1.

It is natural to inquire which legal weight functions on Coxeter graphs give rise to faithful generalized geometric representations $\mathcal{R}_{\Gamma,f}$, and what quotients of $W$ we get by taking the image of $\mathcal{R}_{\Gamma,f}$. The aim of this paper is to investigate these questions and to provide the answers for some classes of Coxeter groups and weight functions.

\subsection{Results.} Our main results are as follows.

\noindent
\textbf{Theorem~\ref{mainsign}.}
Let $\Gamma_f=(V,E,f)$ be a signed Coxeter graph. Then the representation $\mathcal{R}_{\Gamma,f}$ is faithful if and only if $\Gamma_f$ is balanced, i.e., if and only if every cycle in the graph has an even number of $-1$'s.

One direction is true for arbitrary weighted graphs, namely

\noindent
\textbf{Proposition~\ref{mainis}.}
For a balanced weighted Coxeter graph $\Gamma_f=(V,E,f)$ the representation $\mathcal{R}_{\Gamma,f}$ is faithful.

For balanced graphs we associate with the representation $\mathcal{R}_{\Gamma,f}$ a generalized numbers game, similarly to how Mozes numbers game is associated with the (standard) geometric representation, see Section~\ref{BalWCxGrNumGm} below.

\noindent
\textbf{Theorem~\ref{nonfaith}.}
Let $\Gamma_f=(V,E,f)$ be a weighted Coxeter graph. Suppose, the vertices $s_1, s_2,...,s_n$ and edges $\left(s_1,s_2\right)$, $\left(s_2,s_3\right)$,...,$\left(s_n,s_1\right)$ form a cycle for which the product of weights of the edges is equal to $a$, where $a$ is an element of finite order $m$, $m>1$. Then, the representation $\mathcal{R}_{\Gamma,f}$ is not faithful.

Note that Coxeter groups are hopfian (being linear and finitely generated they are residually finite and thus hopfian by a theorem of Mal'cev), and so if $\mathcal{R}_{\Gamma,f}$ is not faithful, then the image of the representation is isomorphic to a proper quotient of $W$ not isomorphic to $W$.

As a partial converse to Theorem~\ref{nonfaith} we prove the following

\noindent
\textbf{Proposition~\ref{onecyclefaith}.}
Let $\Gamma_f=(V,E,f)$ be a weighted cycle with vertices $s_1, s_2,...,s_n$ and edges $\left(s_1,s_2\right)$, $\left(s_2,s_3\right)$,...,$\left(s_n,s_1\right)$,  and the product of weights of the edges is equal to $a$, where $a$ is an element of infinite order. Then, the representation $\mathcal{R}_{\Gamma,f}$ is faithful, i.e., the matrix group generated by matrices $\mathcal{R}_{\Gamma,f}\left(s_i\right)=\omega_i$ (for $1\leqslant i\leqslant n$) is isomorphic to the affine Coxeter group $\tilde{A}_{n-1}$ of the Euclidean type $\tilde{\mathbb A}_{n-1}$, see~\cite{Kac}.

\section{Balanced weighted Coxeter graphs and generalized numbers game}\label{BalWCxGrNumGm}
\bpr\label{mainis}
For a balanced weighted Coxeter graph $\Gamma_f=(V,E,f)$ the representation $\mathcal{R}_{\Gamma,f}$ is faithful.
\epr
\begin{proof}
We associate weights to the vertices of the graph in such a way that the weight of an edge is the ratio of the weights of its ends. It can be done as follows. Take any vertex to be the origin and put its weight to be 1. The weight of any other vertex will be the weight of a path from the origin to this vertex. The graph is balanced and therefore this construction is well defined.

Now consider the standard geometric representation (without weights) and perform the diagonal change of basis: let us multiply each basis vector by the weight of the corresponding vertex. This way we get exactly the representation $\mathcal{R}_{\Gamma,f}$. So, the representation $\mathcal{R}_{\Gamma,f}$ is faithful since it is isomorphic to the standard geometric representation which is faithful.
\end{proof}
Let us illustrate our proof with the following example.
\bex
Consider the simply laced Coxeter graph:\\
\xygraph{
!{<0cm,0cm>;<1cm,0cm>:<0cm,1cm>::}
!{(3,1) }*+{s_1}="1"
!{(5,1) }*+{s_2}="2"
!{(6.6,2) }*+{s_3}="3"
!{(6.6,0)}*+{s_4}="4"
!{(9,1)}*+{s_5}="5"
!{(11,1)}*+{s_6}="6"
"1"-"2" "2"-"4"
"2"-"3"  "4"-"5"
"3"-"5"  "5"-"6"
}\\
Let the group $G$ be the cyclic group $C_2=\{+1,-1\}$. In this case we call the weighted graph ``signed". Consider the following signed graph, which is balanced, as it can be easily seen:
\\
 \xygraph{
!{<0cm,0cm>;<1cm,0cm>:<0cm,1cm>::}
!{(3,1) }*+{s_1}="1"
!{(5,1) }*+{s_2}="2"
!{(6.6,2) }*+{s_3}="3"
!{(6.6,0)}*+{s_4}="4"
!{(9,1)}*+{s_5}="5"
!{(11,1)}*+{s_6}="6"
"1"-"2"^{-1} "2"-"4"^{+1}
"2"-"3"^{-1}  "4"-"5"^{+1}
"3"-"5"^{-1}  "5"-"6"^{+1}
}\\
Now we associate weights to the vertices applying the process described above in the proof of Proposition~\ref{mainis}:\\
\xygraph{
!{<0cm,0cm>;<1cm,0cm>:<0cm,1cm>::}
!{(3,1) }*+{(s_1,+1)}="1"
!{(5.8,1) }*+{(s_2,-1)}="2"
!{(8.5,2) }*+{(s_3,+1)}="3"
!{(8.5,0)}*+{(s_4,-1)}="4"
!{(11.8,1)}*+{(s_5,-1)}="5"
!{(14.7,1)}*+{(s_6,-1)}="6"
"1"-"2"^{-1} "2"-"4"^{+1}
"2"-"3"^{-1}  "4"-"5"^{+1}
"3"-"5"^{-1}  "5"-"6"^{+1}
}\\
According to the weights of vertices we construct the diagonal matrix of the change of basis:
$$
J=\begin{pmatrix}
1 &0 &0 &0 &0 &0\\
0 &-1 &0 &0 &0 &0\\
0 &0 &1 &0 &0 &0\\
0 &0 &0 &-1 &0 &0\\
0 &0 &0 &0 &-1  &0\\
0 &0 &0 &0 &0 &-1\end{pmatrix}\,.
$$
One can check that for $1\leqslant i \leqslant 6$ we have $J\sigma_i J^{-1}=\omega_i$, where $\sigma_i$ are images of generators $s_i$ under the standard geometric representation.
Take for example the matrices $\sigma_2$ and $\omega_2$:
$$
\sigma_2=\begin{pmatrix}
1 &0 &0 &0 &0 &0\\
1 &-1 &1 &1 &0 &0\\
0 &0 &1 &0 &0 &0\\
0 &0 &0 &1 &0 &0\\
0 &0 &0 &0 &1  &0\\
0 &0 &0 &0 &0 &1\end{pmatrix}\,\,,\,\,
\omega_2=\begin{pmatrix}
1 &0 &0 &0 &0 &0\\
-1 &-1 &-1 &1 &0 &0\\
0 &0 &1 &0 &0 &0\\
0 &0 &0 &1 &0 &0\\
0 &0 &0 &0 &1  &0\\
0 &0 &0 &0 &0 &1\end{pmatrix}\,.
$$
We have:
\begin{align*}
J&\sigma_2J^{-1}=\\
&=\begin{pmatrix}
1 &0 &0 &0 &0 &0\\
0 &-1 &0 &0 &0 &0\\
0 &0 &1 &0 &0 &0\\
0 &0 &0 &-1 &0 &0\\
0 &0 &0 &0 &-1  &0\\
0 &0 &0 &0 &0 &-1\end{pmatrix}
\begin{pmatrix}
1 &0 &0 &0 &0 &0\\
1 &-1 &1 &1 &0 &0\\
0 &0 &1 &0 &0 &0\\
0 &0 &0 &1 &0 &0\\
0 &0 &0 &0 &1  &0\\
0 &0 &0 &0 &0 &1\end{pmatrix}
\begin{pmatrix}
1 &0 &0 &0 &0 &0\\
0 &-1 &0 &0 &0 &0\\
0 &0 &1 &0 &0 &0\\
0 &0 &0 &-1 &0 &0\\
0 &0 &0 &0 &-1  &0\\
0 &0 &0 &0 &0 &-1\end{pmatrix}=\\
&=\begin{pmatrix}
1 &0 &0 &0 &0 &0\\
-1 &-1 &-1 &1 &0 &0\\
0 &0 &1 &0 &0 &0\\
0 &0 &0 &1 &0 &0\\
0 &0 &0 &0 &1  &0\\
0 &0 &0 &0 &0 &1\end{pmatrix}=\omega_2\,.
\end{align*}

\eexa

We are now going to define a generalized version of the numbers game, corresponding to the faithful representation $\mathcal{R}_{\Gamma,f}$, which is associated to a balanced weighted Coxeter graph $\Gamma_f$. We begin with the brief description of the classical numbers game played on Coxeter graph of a Coxeter system $(W,S)$ according to~\cite{BB}. The point of this game is that it gives a combinatorial model of the Coxeter group $(W,S)$ where group elements correspond to positions and reduced decompositions correspond to legal play sequences.

Before proceeding we have to choose for each ordered pair of generators $(s,s')$ such that $m(s,s')\geqslant 3$ a real number $k_{s,s'}>0$ such that
$$
\left\{ \begin{array}{cc}
{k_{s,s'}k_{s',s}=4\cos^2{\pi\over m(s,s')}\,\,\,\,,\,\,}      &  {\textrm{if}\,\,\,m(s,s')\neq\infty} \\
{k_{s,s'}k_{s',s}\geqslant 4\,\,\,\,\,\,\qquad\qquad\,\,\,\,,\,\,}          & {\textrm{if}\,\,\,m(s,s')=\infty}
\end{array}\right.
$$
These numbers, that we refer to as weights, remain fixed once chosen. The edge weights can of course always be chosen symmetrically:
$$
k_{s,s'}=k_{s',s}=2\cos{\pi\over m(s,s')} \,\,.
$$
For $m(s,s')=3$, $k_{s,s'}=k_{s',s}=1$. For $m(s,s')\in\{4,6\}$ the asymmetric choice of weights has an advantage since all weights are integers: for $m(s,s')=4$, $k_{s,s'}=2$, $k_{s',s}=1$; for $m(s,s')=6$, $k_{s,s'}=3$, $k_{s',s}=1$. For $m(s,s')=\infty$ the weights also can be taken integers: $k_{s,s'}=k_{s',s}=2$.

The starting position for the game can be any distribution $s\mapsto p_s$ of real numbers $p_s$ to the nodes $s\in S$ of the Coxeter graph. A position is called positive if $p_s>0$ for all  $s\in S$. We think of a game position as of column-vector $p=\left(p_{s_1},p_{s_2},...\right)^t\in\mathbb R^{|S|}$. A game position is also called a {\it configuration}. The special position with $p_s=1$ for all $s\in S$ is called the unit position and denoted {\bf 1}.

Moves are defined as follows. A firing of node $s$ changes a position $p\in \mathbb R^{|S|}$ in the following way:
\begin{itemize}
\item switch sign of the value at $s$, $p_s$,
\item add $k_{s,s'}p_s$ to the value at each neighbor $s'$ of $s$,
\item leave all other values unchanged.

\end{itemize}
Such a move is called positive if $p_s>0$ and negative if $p_s<0$. A positive game is one that is played with positive moves from a given starting position, and similarly for a negative game. A (positive, negative) play sequence is a word $s_1s_2\ldots s_{\ell}$ ($s_i\in S$) recording a (positive, negative) game in which $s_1$ is fired first, then $s_2$, then $s_3$ and so on.

Easy to see that the firing of the node $s$ is the action of the matrix $\sigma_s^t$ (on the column vector of the position $p\in\mathbb R^{|S|}$, where the matrix $\sigma_s$ is the image of the generator $s$ under the geometric representation.

Suppose a starting position $p\in\mathbb R^{|S|}$ is given. Then every play sequence $s_1s_2\ldots s_{\ell}$ will by composition of mappings $\sigma^t_i$ lead to some other position denoted by $p^{s_1s_2\ldots s_{\ell}}$. Let $\mathcal P_p\subset \mathbb R^{|S|}$ denote the set of all positions that can be reached this way. The following theorem is the main theorem concerning the numbers game given in~\cite{BB}.
\begin{thm}\label{BBNG} Consider play sequences starting from a certain positive position $p\in\mathbb R^{|S|}_+$.
\begin{enumerate}
\item Two play sequences $s_1s_2\ldots s_{\ell}$ and $s'_1s'_2\ldots s'_{\ell}$ ($s_i,s'_i\in S$) lead to the same position (i.e., $p^{s_1s_2\ldots s_k}=p^{s'_1s'_2\ldots s'_{\ell}}$) if and only if  $s_1s_2\ldots s_{\ell}=s'_1s'_2\ldots s'_k$ as elements of the Coxeter group $W$.
\item The induced mapping $w\mapsto p^w$ is a bijection $W\rightarrow \mathcal P_p$.
\item Let $D_R(w)$ denote the right descent set of $w$.\\ Then $D_R(w)=\left\{s\in S\,:\,\textrm{the s-coordimate of}\,\, p^w\,\,\textrm{is negative}\right\}$
\item The play sequence $s_1s_2\ldots s_{\ell}$ is positive if and only if $s_1s_2\ldots s_{\ell}$ is a reduced decomposition.
\end{enumerate}
\end{thm}
Now we describe our generalization of the classical numbers game. We choose as the edge weights the complex numbers $k_{s,s'}$ satisfying the conditions
\begin{enumerate}
\item For any pair of different generators $(s,s')$
$$
\left\{ \begin{array}{cc}
{k_{s,s'}k_{s',s}=4\cos^2{\pi\over m(s,s')}\,\,\,\,,\,\,}      &  {\textrm{if}\,\,\,m(s,s')\neq\infty} \\
{k_{s,s'}k_{s',s}\geqslant 4\,\,\,\,\,\,\qquad\qquad\,\,\,\,,\,\,}          & {\textrm{if}\,\,\,m(s,s')=\infty}
\end{array}\right.
$$
\item For any pair of different generators $(s,s')$ the number $k_{s,s'}$ can be decomposed as
$$
k_{s,s'}=f_{s,s'}\ell_{s,s'}
$$
in such a way that for any two vertices $s$ and $s'$ connected with an edge
$f_{s,s'}f_{s',s}=1$ and $\ell_{s,s'}$, $\ell_{s',s}$ are real positive.
(This means that the numbers $\ell_{s,s'}$ satisfy the condition (1) like the numbers $k_{s,s'}$.)
\item
For any cycle with vertices $s_1$, $s_2$, ..., $s_t$
$$
f_{s_1,s_2}\cdot f_{s_2,s_3}\cdots
 f_{s_{t-1},s_t}\cdot f_{s_{t},s_1}=1\,.
$$
\end{enumerate}
Unlike the classical numbers game we do not require $k_{s,s'}$ to be positive real numbers. Instead of the positivity we require conditions (2) and (3) which are weaker. The classical numbers game can be seen as a particular case of this one where $f_{s_i,s_j}=1$ for all edges $\left(s_i,s_j\right)$.

Now we have a balanced weighted Coxeter graph $\Gamma_f$ with a legal weight function $(s,s')\mapsto f_{s,s'}$ (see Definition~\ref{lwf}). We choose the vertex $s_1$ to be the origin and associate to each vertex $s_i$ the weight $wt(s_i)$ of a path from the origin to $s_i$. Recall that the weight of a path is the product of weights  $f_{s_i,s_j}$ where $(s_i, s_j)$ are the edges of which this path consists. A move associated to a vertex $s$ changes a configuration $p$ exactly as above. The only difference is that instead of the notion of a positive (negative) move we need to use a {\it pseudo-positive  (pseudo-negative)} move: a move associated to a vertex $s$ is called {\it pseudo-positive  (pseudo-negative)} if $wt(s)p_s$ is a positive (negative) real number.

 Let $J$ be the diagonal matrix constructed in the proof of Proposition~\ref{mainis}. Let $p\in\mathbb R^{|S|}$ be a position, i.e., a column-vector. It follows immediately from the definition of a pseudo-positive move that the move $s$ is pseudo-positive for $p$ in the described above generalized numbers game associated to the balanced weighted Coxeter graph $\Gamma_f$ if and only if this move is positive for the position $Jp$ in the classical numbers game associated to the same graph $\Gamma$.

 Let $\sigma_s$ and $\omega_s$ be the matrices which correspond to the generator $s$ under the standard geometric representation and the generalized geometric representation $\mathcal{R}_{\Gamma,f}$ respectively. Then $J\sigma_s J^{-1}=\omega_s$, see the proof of Proposition~\ref{mainis}. Hence, $J\cdot\left(\left(\omega_s\right)^t p\right)=\left(\sigma_s\right)^t\cdot(Jp)$. Thus, a play sequence $s_1s_2\cdots s_{\ell}$ is pseudo-positive for the generalized game associated to the balanced weighted Coxeter graph $\Gamma_f$ if and only if this sequence is positive for the classical game associated to the same Coxeter graph $\Gamma$.

 In view of these simple observations the main Theorem~\ref{BBNG} concerning the numbers game given in~\cite{BB} and cited above can be formulated for this generalized numbers game too as follows:
\begin{thm} Consider play sequences starting from a certain position $p\in\mathbb R^{|S|}$ such that $Jp\in\mathbb R^{|S|}_+$. Denote by $\tilde{\mathcal P}_p\subset \mathbb R^{|S|}$ all the positions which can be reached by moves of the generalized numbers game starting with $p$.
\begin{enumerate}
\item Two play sequences $s_1s_2\ldots s_{\ell}$ and $s'_1s'_2\ldots s'_{\ell}$ ($s_i,s'_i\in S$) lead to the same position (i.e., $p^{s_1s_2\ldots s_k}=p^{s'_1s'_2\ldots s'_{\ell}}$) if and only if  $s_1s_2\ldots s_{\ell}=s'_1s'_2\ldots s'_k$ as elements of the Coxeter group $W$.
\item The induced mapping $w\mapsto p^w$ is a bijection $W\rightarrow \tilde{\mathcal P}_p$.
\item Let $D_R(w)$ denote the right descent set of $w$.\\ Then $D_R(w)=\left\{s\in S\,:\,\textrm{the s-coordimate of}\,\, Jp^w\,\,\textrm{is negative}\right\}$
\item The play sequence $s_1s_2\ldots s_{\ell}$ is pseudo-positive if and only if $s_1s_2\ldots s_{\ell}$ is a reduced decomposition.
\end{enumerate}
\end{thm}

\section{Non-balanced weighted Coxeter graphs}
 In this Section we present an obstruction to faithfulness of the generalized geometric representation. Namely, we prove the following
\begin{thm}\label{nonfaith}
Let $\Gamma=(V,E,f)$ be a weighted Coxeter graph. Suppose, the vertices $s_1, s_2,...,s_n$ and edges $\left(s_1,s_2\right)$, $\left(s_2,s_3\right)$,...,$\left(s_n,s_1\right)$, form the cycle for which the product of weights of the edges is equal to $a$, where $a$ is an element of finite order $m$, $m>1$. Then, the representation $\mathcal{R}_{\Gamma,f}$ is not faithful.
\end{thm}
The proof of Theorem~\ref{nonfaith} uses the following observation.
\bpr\label{gather}
Let $\Gamma=(V,E)$ be a cycle with $n$ vertices $s_1$,...,$s_n$ and directed edges $e_1=\left(s_1,s_2\right)$, $e_2=\left(s_2,s_3\right)$,...,$e_{n-1}=\left(s_{n-1},s_n\right)$, $e_n=\left(s_n,s_1\right)$. Let $f,h:\bar E\rightarrow \mathbb C\setminus\{0\}$ be two legal weight functions  defined as follows:
$$f\left(e_i\right)=a_i\,,\,1\leqslant i\leqslant n\,\,,$$
$$h\left(e_i\right)=1\,,\,1\leqslant i\leqslant n-1\,\,,\,\,h\left(e_n\right)=a_1a_2\cdots a_n\,.$$
Then representations $\mathcal{R}_{\Gamma,f}$ and $\mathcal{R}_{\Gamma,h}$ are isomorphic.
\epr
\begin{proof}
Define the diagonal $n\times n$ matrix:
$$
J=\begin{pmatrix}
1 &0   &0      &0 &\cdots &0\\
0 &a_1 &0      &0 &\cdots &0\\
0 &0   &a_1a_2 &0 &\cdots &0\\
\cdots &\cdots &\cdots &\cdots &\cdots &\cdots\\
0 &\cdots &\cdots  &0 &a_1a_2\cdots a_{n-2} &0\\
0 &\cdots &\cdots &0 &0 &a_1a_2\cdots a_{n-2}a_{n-1}\end{pmatrix}\,.
$$
It can be easily seen that $J\mathcal{R}_{\Gamma,f}\left(s_i\right)J^{-1}=\mathcal{R}_{\Gamma,h}\left(s_i\right)$ for $1\leqslant i \leqslant n$.
\end{proof}

Let us illustrate it by the following:
\bex Consider the graph

$\qquad\qquad\qquad\qquad$ \xygraph{
!{<0cm,0cm>;<1cm,0cm>:<0cm,1cm>::}
!{(5,1) }*+{s_1}="1"
!{(7,2) }*+{s_2}="2"
!{(9,1) }*+{s_3}="3"
!{(7,0)}*+{s_4}="4"
"1"-"2"^{a} "1"-"4"^{d}
"2"-"3"^{b} "3"-"4"^{c}
}

We go along the graph in the clockwise direction starting from the vertex which corresponds to $s_1$ and associated weights to the vertices:

$\qquad\qquad\qquad\qquad$ \xygraph{
!{<0cm,0cm>;<1cm,0cm>:<0cm,1cm>::}
!{(5,1) }*+{(s_1,1)}="1"
!{(7,2) }*+{(s_2,a)}="2"
!{(9,1) }*+{(s_3,ab)}="3"
!{(7,0)}*+{(s_4,abc)}="4"
"1"-"2"^{a} "1"-"4"^{d}
"2"-"3"^{b} "3"-"4"^{c}
}

Here $a,b,c,d$ are elements of some abelian group. The images of generators under the corresponding representation ($s_i\mapsto\omega_i$) are:
$$\omega_1=\begin{pmatrix}
-1 &a &0 &d^{-1}\\
0 &1 &0 &0 \\
0 &0 &1 &0 \\
0 &0 &0 &1 \end{pmatrix}\,\,,\,\,\,
\omega_2=\begin{pmatrix}
1 &0 &0 &0\\
a^{-1} &-1 &b &0 \\
0 &0 &1 &0 \\
0 &0 &0 &1 \end{pmatrix}\,\,, $$

$$\omega_3=\begin{pmatrix}
1 &0 &0 &0 \\
0 &1 &0 &0 \\
0 &b^{-1} &-1 &c \\
0 &0 &0 &1 \end{pmatrix},\,
\omega_4=\begin{pmatrix}
1 &0 &0 &0 \\
0 &1 &0 &0 \\
0 &0 &1 &0 \\
d &0 &c^{-1} &-1\end{pmatrix}.$$
Define the matrix
$$
J=\begin{pmatrix}
1 &0 &0 &0 \\
0 &a &0 &0 \\
0 &0 &ab &0 \\
0 &0 &0 &abc \end{pmatrix}\,.
$$
The matrix $J$ is the matrix of the change of basis.
We have
$$
J\omega_1J^{-1}=\tilde{\omega}_1=\begin{pmatrix}
-1 &1 &0 &(abcd)^{-1} \\
0 &1 &0 &0 \\
0 &0 &1 &0 \\
0 &0 &0 &1 \end{pmatrix}\,\,,\,\,\,
J\omega_2J^{-1}=\tilde{\omega}_2=\begin{pmatrix}
1 &0 &0 &0  \\
1 &-1 &1 &0  \\
0 &0 &1 &0 \\
0 &0 &0 &1 \end{pmatrix}\,, $$
$$
J\omega_3J^{-1}=\tilde{\omega}_3=\begin{pmatrix}
1 &0 &0 &0  \\
0 &1 &0 &0  \\
0 &1 &-1 &1  \\
0 &0 &0 &1 \end{pmatrix}\,\,,\qquad\quad
J\omega_4J^{-1}=\tilde{\omega}_4=\begin{pmatrix}
1 &0 &0 &0  \\
0 &1 &0 &0  \\
0 &0 &1 &0  \\
abcd &0 &1 &-1 \end{pmatrix}\,.
$$
The representation $s_i\mapsto\tilde{\omega}_i$ corresponds to the graph where all edges have weights $1$ except one edge which has the weight $abcd$ which is the product of the weights of the edges of the original graph:

$\qquad\qquad\qquad\qquad$ \xygraph{
!{<0cm,0cm>;<1cm,0cm>:<0cm,1cm>::}
!{(5,1) }*+{s_1}="1"
!{(7,2) }*+{s_2}="2"
!{(9,1) }*+{s_3}="3"
!{(7,0)}*+{s_4}="4"
"1"-"2"^{1} "1"-"4"^{abcd}
"2"-"3"^{1} "3"-"4"^{1}
}

If we start with another vertex, we gather the product of weights on another edge.
\eexa

\noindent
\textbf{Proof of Theorem~\ref{nonfaith}.}
 Denote by $t_i$ permutation matrices:
 $$
 t_1= \begin{pmatrix}
 0 &1 &0 &0 &\cdots &0\\
 1 &0 &0 &0  &\cdots &0\\
 0 &0 &1 &0 &\cdots &0\\
 \cdots &\cdots &\cdots &\cdots &\cdots &\cdots\\
 0 &\cdots &\cdots &0 &1 &0\\
 0 &\cdots &\cdots &0 &0 &1
 \end{pmatrix}\,\,,\,\,
 t_2= \begin{pmatrix}
 1 &0 &0 &0 &\cdots &0\\
 0 &0 &1 &0  &\cdots &0\\
 0 &1 &0 &0 &\cdots &0\\
 \cdots &\cdots &\cdots &\cdots &\cdots &\cdots\\
 0 &\cdots &\cdots &0 &1 &0\\
 0 &\cdots &\cdots &0 &0 &1
 \end{pmatrix}\,\,,...
$$
 $$
 t_{n-1}= \begin{pmatrix}
 1 &0  &0 &\cdots &0\\
 0 &1  &0  &\cdots &0\\
 \cdots &\cdots  &\cdots &\cdots &\cdots\\
 0 &\cdots  &0 &0 &1\\
 0 &\cdots  &0 &1 &0
 \end{pmatrix}\,\,,
 $$
 and the following monomial matrix $A$
 $$
 A= \begin{pmatrix}
 0 &0 &0 &\cdots &0 &a^{-1}\\
 0 &1 &0 &0  &\cdots &0\\
 0 &0 &1 &0 &\cdots &0\\
 \cdots &\cdots &\cdots &\cdots &\cdots &\cdots\\
 0 &\cdots &\cdots &0 &1 &0\\
 a &0 &\cdots &0 &0 &0
 \end{pmatrix}\,\,.
$$

Using Proposition~\ref{gather}, without loss of generality we can assume that the cycle with vertices $s_1, s_2,...,s_n$ have the weight $a$ on the edge $\left(s_n,s_1\right)$ and 1's on other edges. Consider the parabolic subgroup $U$ of $W$ generated by $s_1, s_2,...,s_n$ and consider the representation $\tau$ of $U$, which is the restriction of $\mathcal{R}_{\Gamma,f}$ on $U$.
$$\tau(s_1)=\omega_1=\begin{pmatrix}
-1 &1 &0 &\cdots  &0 &a^{-1}\\
0 &1 &0 &0 &\cdots  &0\\
0 &0 &1 &0 &\cdots  &0\\
 \cdots &\cdots &\cdots &\cdots &\cdots &\cdots\\
 0 &\cdots &\cdots &0 &1 &0\\
0 &0 &0 &\cdots  &0 &1 \end{pmatrix}\,\,,
$$
$$
\tau(s_2)=\omega_2=\begin{pmatrix}
1 &0 &0 &0 &\cdots  &0\\
1 &-1 &1 &0 &\cdots  &0\\
0 &0 &1 &0 &\cdots  &0\\
 \cdots &\cdots &\cdots &\cdots &\cdots &\cdots\\
 0 &\cdots &\cdots &0 &1 &0\\
0 &0 &0 &\cdots  &0 &1 \end{pmatrix}\,\,, $$
$$\cdots\cdots\cdots\cdots\cdots\cdots\cdots\cdots\cdots\cdots\cdots\cdots\cdots\cdots\cdots\cdots $$
$$\tau(s_{n-1})=\omega_{n-1}=\begin{pmatrix}
1 &0 &0 &0 &\cdots  &0\\
0 &1 &0 &0 &\cdots  &0\\
\cdots &\cdots &\cdots &\cdots &\cdots &\cdots\\
0 &\cdots  &0 &1 &-1 &1 \\
0 &\cdots  &0 &0 &0 &1 \end{pmatrix}\,\,,
$$
$$
\tau(s_{n})=\omega_{n}=\begin{pmatrix}
1 &0 &0 &0 &\cdots  &0\\
0 &1 &0 &0 &\cdots  &0\\
\cdots &\cdots &\cdots &\cdots &\cdots &\cdots\\
0 &0 &\cdots  &0 &1 &0 \\
a &0 &\cdots  &0 &1 &-1 \end{pmatrix}\,\,.
$$

 Define the matrix $J$ in the following way: $J$ has 1's in positions $(i,i)$, $-1$'s in positions $(i+1,i)$ and $-a^{-1}$ in position $(1,n)$:
$$
J=\begin{pmatrix}
1 &0 &0 &\cdots &0 &-a^{-1}\\
-1 &1 &0 &\cdots &0 &0 \\
0 &-1 &1 &\cdots &0 &0 \\
\cdots &\cdots &\cdots &\cdots &\cdots &\cdots\\
0 &\cdots  &0 &-1 &1 &0\\
0 &\cdots &0 &0 &-1 &1 \end{pmatrix}\,.
$$
  We have $J\omega_i J^{-1}=t_i$ for $1\leqslant i\leqslant n-1$ and $J\omega_n J^{-1}=A$. Since the element $a$ is of finite order and the matrices $t_1, t_2,...,t_{n-1},A$ are monomial the matrix group $\langle t_1, t_2,...,t_{n-1},A\rangle$ is a group of monomial matrices with finite number of different entries and so it is a finite group (isomorphic to $C_m^{n-1}\rtimes S_n$ of order $m^{n-1}n!$ where $m$ is the order of the element $a$ and $C_m$ is the multiplicative cyclic group). Therefore the matrix group $\langle \omega_1, \omega_2,...,\omega_{n-1},\omega_n\rangle$ is also finite because it is conjugate to the finite group $\langle t_1, t_2,...,t_{n-1},A\rangle$ (the conjugating matrix is $J$). So, the representation is not faithful since its image is a finite matrix group while the Coxeter group which corresponds to a cycle is an affine group of type A and is infinite. \hfill $\square$

\bex
Consider the cycle with four vertices.

$\qquad\qquad\qquad\qquad\qquad$ \xygraph{
!{<0cm,0cm>;<1cm,0cm>:<0cm,1cm>::}
!{(5,1) }*+{s_1}="1"
!{(7,1.5) }*+{s_2}="2"
!{(9,1) }*+{s_3}="3"
!{(7,0.5)}*+{s_4}="4"
"1"-"2"^{1} "1"-"4"^{a}
"2"-"3"^{1}
"3"-"4"^{1}
}

The matrices $\omega_i$ are:
$$\omega_1=\begin{pmatrix}
-1 &1 &0 &a^{-1}\\
0 &1 &0 &0 \\
0 &0 &1 &0 \\
0 &0 &0 &1 \end{pmatrix}\,\,,\,\,\,
\omega_2=\begin{pmatrix}
1 &0 &0 &0\\
1 &-1 &1 &0 \\
0 &0 &1 &0 \\
0 &0 &0 &1 \end{pmatrix}\,\,, $$

$$\omega_3=\begin{pmatrix}
1 &0 &0 &0\\
0 &1 &0 &0 \\
0 &1 &-1 &1 \\
0 &0 &0 &1 \end{pmatrix}\,\,,\,\,\,
\omega_4=\begin{pmatrix}
1 &0 &0 &0\\
0 &1 &0 &0 \\
0 &0 &1 &0 \\
a &0 &1 &-1 \end{pmatrix}\,\,.$$
Define
$$
J=\begin{pmatrix}
1 &0 &0 &-a^{-1}\\
-1 &1 &0 &0 \\
0 &-1 &1 &0 \\
0 &0 &-1 &1 \end{pmatrix}\,.
$$
Now we have
$$
J\omega_1 J^{-1}=\begin{pmatrix}
0 &1 &0 &0\\
1 &0 &0 &0 \\
0 &0 &1 &0 \\
0 &0 &0 &1 \end{pmatrix}=t_1\,,\,J\omega_2 J^{-1}=\begin{pmatrix}
1 &0 &0 &0\\
0 &0 &1 &0 \\
0 &1 &0 &0 \\
0 &0 &0 &1 \end{pmatrix}=t_2\,,
$$
$$
J\omega_3 J^{-1}=\begin{pmatrix}
1 &0 &0 &0\\
0 &1 &0 &0 \\
0 &0 &0 &1 \\
0 &0 &1 &0 \end{pmatrix}=t_3\,,\,J\omega_4 J^{-1}=\begin{pmatrix}
0 &0 &0 &a^{-1}\\
0 &1 &0 &0 \\
0 &0 &1 &0 \\
a &0 &0 &0 \end{pmatrix}=A\,.
$$
The group of $4\times 4$ matrices generated by $t_1, t_2, t_3, A$ is isomorphic to a semi-direct product of the group $C_m\times C_m\times C_m$ and $S_4$ where $m$ is the order of the element $a$, $C_m$ and $S_4$, as usual, denote the cyclic multiplicative group of order $m$ and the symmetric group of permutations of the set $\{1,2,3,4\}$ respectively.
\eexa
As a partial converse of Theorem~\ref{nonfaith} we have the following:
\bpr\label{onecyclefaith}
Let $\Gamma=(V,E,f)$ be a weighted cycle with vertices $s_1, s_2,...,s_n$ and the product of weights of the edges is equal to $a$, where $a$ is an element of infinite order. Then, the representation $\mathcal{R}_{\Gamma,f}$ is faithful, i.e., the matrix group generated by matrices $\mathcal{R}_{\Gamma,f}\left(s_i\right)=\omega_i$ (for $1\leqslant i\leqslant n$) is isomorphic to the affine Coxeter group $\tilde{A}_{n-1}$ of the Euclidean type $\tilde{\mathbb A}_{n-1}$, see~\cite{Kac}.
\epr

\begin{proof}

We define $\omega_i$ for $1\leq\i\leq n$ as an $n\times n$ matrix, where denote
the  entry of $\omega_i$ in the $j$-th row, and $k$-th column by
$a^{(i)}_{j,k}$. Then, $a^{(i)}_{k,k}=1$ for every $k\neq i$,
$a^{(i)}_{k,l}=0$ for $k\neq i$ and $k\neq
l$, $a^{(i)}_{i,i}=-1$, $a^{(i)}_{i,i-1}=1$ for $2\leq i\leq n$,
$a^{(i)}_{i,i+1}=1$ for $1\leq i\leq n-1$,
$a^{(1)}_{1,n}=a^{-1}$, $a^{(n)}_{n,1}=a$, and $a^{(i)}_{i,j}=0$, if
$i-j\notin \{0,1,-1\}$ modulo $n$.

Let $J$ be an $n\times n$ matrix where the entry of $J$ in the $i$-th
row, and $j$-th column is denoted by $b_{i,j}$. Define $b_{i,i}=1$ for
$1\leq i\leq n$, $b_{i+1,i}=-1$ for $1\leq i\leq n-1$, $b_{1,n}=-a^{-1}$,
and $b_{i,j}=0$ for $i-j\notin \{0, 1\} ~~mod ~~n$. Let $t_i$ be
$J\omega_{i}J^{-1}$ for every $1\leq i\leq n-1$, and let A be
$J\omega_{n}J^{-1}$.
Then we have, $t_i$ is a permutation matrix for $1\leq i\leq n-1$, and $A$
is a matrix which satisfies $A_{i,i}=1$ for $2\leq i\leq n-1$,
$A_{1,1}=A_{n,n}=0$, $A_{1,n}=a^{-1}$, $A_{n,1}=a$, and $A_{i,j}=0$
if $|i-j|\notin \{0, n-1\}$, where $A_{i,j}$ denotes the entry of $A$ in
the $i$-th row and $j$-th column.

Let $G=J\Omega J^{-1}$. It is clear that $G=\langle t_1, t_2, ..., t_{n-1},
A\rangle$
Define a map $\varphi:G\rightarrow \tilde{A}_{n-1}$, such that
$\varphi(t_i)=s_i$ for $1\leq i\leq n-1$, and $\varphi(A)=s_0$, where
$s_0, s_1, ..., s_{n-1}$ are the standard Coxeter generators of
$\tilde{A}_{n-1}$.

The matrix $A$ can be written as
$tD_0$, where $t$ is an $n\times n$ permutation matrix, and $D_0$ is an $n\times n$
diagonal matrix,
which satisfies $D_{1,1}=a$, $D_{i.i}=1$ for $2\leq i\leq n-1$, and
$D_{n,n}=a^{-1}$, where $D_{i,i}$ is the element in the $i$-th row and in
the $i$-th column of $D_0$. In particular $\det(D_0)=1$. Since, $t_i$ are
$n\times n$ permutation matrices for  every $1\leq  i\leq n-1$, every generator
of $G$ can be written as $XD$, where $X$ is a $n\times n$ permutation matrix and
$D$ is a $n\times n$ diagonal matrix, which entries in the diagonal are integer
powers of the element $a$, such that $\det(D)=1$. It is easy to show: if
$a_1$ and $a_2$ are $n\times n$ permutation matrices, and $D_1$ and $D_2$ are
$n\times n$ diagonal matrices, such that all the elements in the diagonal of
$D_1$ and in the diagonal of $D_2$ are integer powers of $a$, and
$\det(D_1)=\det(D_2)=1$, then the product $(a_1D_1)(a_2D_2)=a_1a_2D_3$,
where $D_3$ is also a diagonal matrix which entries are integer powers of
$a$, and $\det(D_3)=1$. Hence, every element of $G$ is a product of a $n\times n$
permutation matrix with a diagonal matrix which diagonal entries are
integer powers of the element $a$ such that the determinant of the
diagonal matrix is equal to 1. According to \cite{BB}, every element in
$\tilde{A}_{n-1}$ is a permutation $\theta$ of $\mathbb{Z}$, such that
$\sum_{i=1}^n\theta(i)=\frac{n(n+1)}{2}$, and $\theta(i+n)=\theta(i)+n$.
Hence, such a permutation can be written as $pq$, where $p$ is a
permutation of the elements $\{1,2,...,n\}$, and $q$ is a permutation of
$\mathbb{Z}$ which satisfies $q(i)=i+nk_i$, where $\sum_{i=1}^n k_i=0$.
If $a_1$ is a $n\times n$ permutation matrix corresponds to a permutation
$p_1$, $a_2$ is a permutation matrix  corresponds to a permutation $p_2$,
$D_1$ and $D_2$ are diagonal matrix, such that $D_{i,i}^{(j)}=k_i^{(j)}$
for $1\leq i\leq n$, and $1\leq j\leq 2$, where $D_{i,i}^{(j)}$ denotes
the entry in the $i$-th row and $i$-th column of the matrix $D_j$. Define
$q_1$ and $q_2$ permutations of $\mathbb{Z}$ such that
$q_j(i)=i-nk_i^{(j)}$, then $(p_1q_1)(p_2q_2)=p_1p_2q_3$, such that
$q_3(i)=i-nk_i^{(3)}$, where $k_i^{(3)}$ is the entry in the $i$-th row
and in the $i$-th column of $D_3$. Thus, there is a one-to-one
correspondence between the elements of $\tilde{A}_{n-1}$ and the
elements of $G$. Since $J\Omega J^{-1}=G$, $\Omega$ is isomorphic to
$\tilde{A}_{n-1}$ too.
\end{proof}

Let us now consider an important particular case of signed Coxeter graphs.
Let $(W,S)$ be a simply laced Coxeter system and let $\Gamma=(V,E)$ be its Coxeter graph. Let $f\,:\,E\rightarrow C_2=\{+1,-1\}$ be an arrangement of signs on the edges and $\Gamma_f=(V,E,f)$ be the signed Coxeter graph. Signed graphs with relations to Coxeter groups are  a well known and studied subject, see for example~\cite{CST}. Also, signed graphs appear in~\cite{AST}.
\begin{thm}\label{mainsign}
Let $\Gamma_f=(V,E,f)$ be a signed Coxeter graph. Then the representation $\mathcal{R}_{\Gamma,f}$ is faithful if and only if $\Gamma_f$ is balanced, i.e., the product of signs along any cycle is $+1$ or in other words any cycle has even number of $-1$'s.
\end{thm}
\begin{proof} For the proof we just have to combine Proposition~\ref{mainis} and Theorem~\ref{nonfaith}.
\end{proof}
\section{Conclusions}
In this paper we define a notion of the legal weight function and a certain generalization of the standard geometric representation of a Coxeter group which is associated to such a function. We give two sufficient conditions for faithfulness and non-faithfulness of this representation in terms of certain properties of the weight function. This research can be continued in various directions. We list here several problems which remain open and seem to be interesting.

\noindent
\textbf{Problem 1.} Is the converse of Theorem~\ref{nonfaith} true? I. e. is it true that the representation $\mathcal{R}_{\Gamma,f}$ is faithful if and only if there is no cycle such that the product of weights along this cycle is an element of finite order greater than one?

\noindent
\textbf{Problem 2.} What can be said about the numbers game associated to the weighted Coxeter graph $\Gamma_f$ when the representation $\mathcal{R}_{\Gamma,f}$ is faithful but the weighted graph $\Gamma_f$ is not balanced?

\noindent
\textbf{Problem 3.} Let the representation $\mathcal{R}_{\Gamma,f}$ be non-faithful. Describe the quotient of the Coxeter group $W$ that we get in this case. Can we define the generalized numbers game in this case?  Is this game related to this quotient in any way? Can any quotient be obtained by finding a suitable legal weight function?

The list of problems can be continued.

{\bf Acknowledgements.} This work was started when the second author was a postdoc at the Department of Mathematics at the University of Geneva; the second author would like to thank the department and especially Prof. Tatiana Smirnova-Nagnibeda
for the hospitality and introducing him to the subject of this paper. Also, the second author is grateful to Prof. Vadim E. Levit for the interesting and helpful discussion of the subject of this paper. The authors are grateful to the reviewers for helpful comments.

\end{document}